\documentclass[12pt]{article}
\usepackage{amsmath,amssymb,amsfonts,latexsym,graphicx}

\catcode `\@=11 \@addtoreset{equation}{section}

\catcode `\@=12
  \newcommand{\cP}{{\mathcal P}}
\newcommand{\cT}{{\mathcal T}} \newcommand{\cC}{{\mathcal C}}

\newcommand{\cH}{\mathcal{H}}

\newenvironment{proof}{\noindent {\bf Proof --}}{\hfill$\square$ \vspace{3mm}\endtrivlist}
\newtheorem{ttt}{\bfseries{Theorem}}[section]
\newtheorem{ddd}{\bfseries{Definition}}
\newtheorem{ppp}{\bfseries{Proposition}}[section]
\newtheorem{lelele}{\bfseries{Lemma}}[section]
\newtheorem{ccc}{\bfseries{Corollary}}[section]
\newtheorem{rem}{\bfseries{Remark}}[section]
\numberwithin{equation}{section}

\catcode `\@=11 \@addtoreset{equation}{section}
\catcode `\@=12
\begin{document}

\thispagestyle{empty}

\vspace*{2cm}

\begin{center}
{\Large \bf Towards Generalized Riesz Systems Theory}   \vspace{2cm}\\

{\large A. Kamuda}\\
AGH University of Science and Technology,  \\ 30-059 al. Mickiewicza 30, Krak\'{o}w, Poland\\
e-mail: kamuda@agh.edu.pl

\vspace{2mm}

{\large S. Ku\.{z}el}\\
AGH University of Science and Technology,  \\ 30-059 al. Mickiewicza 30, Krak\'{o}w, Poland\\
e-mail: kuzhel@agh.edu.pl
\end{center}

%\vspace*{2cm}

\begin{abstract}
\noindent Pseudo-Hermitian Hamiltonians  have recently become a field of wide investigation. Originally, the Generalized Riesz Systems (GRS) have been introduced as an auxiliary tool in this theory. In contrast, the current paper,  GRSs are analysed  in terms  of basis theory. The relationship between semi-regular sequences and GRSs is provided.   Various characterizations of GRSs are discussed.
\end{abstract}

%\vspace{2cm}



\vfill

\newpage

\section{Introduction}\label{sec1}
Theory of non self-adjoint operators  attracts a steady interests in various fields of mathematics and physics,
see, e.g., \cite{BGSZ} and the reference therein.  This interest grew considerably due to the recent progress in
theoretical physics of  $\mathcal{P}\mathcal{T}$-symmetric (pseudo-Hermitian) Hamiltonians \cite{Bender, BenderSI, Most}.
Studies of pseudo-Hermitian operators carried out in \cite{Krejcirik, Mit, Siegl} show that, even if the eigenvalues of a Hamiltonian are real,
the Riesz basis property of its eigenstates is, in many cases, lost. 

Such kind of phenomenon is typical for  $\mathcal{PT}$-symmetric Hamiltonians and it gives rise to a natural problem:
\emph{How to define a suitable generalization of the Riesz bases concept which can be useful in theory of pseudo-Hermitian 
Hamiltonians?}

One of generalizations was proposed by Davies \cite{Davies}: the concept of tame and wild sequences.  
Each basis  is a tame sequence.  The tameness of eigenstates of non self-adjoint operators $H$ with a purely discrete real spectrum 
allows one to discover additional properties of $H$. In 
particular, a polynomially bounded behavior of the corresponding resolvent was established in \cite[Theorem 3]{Davies}.
However, in major part, eigenstates of pseudo-Hermitian Hamiltonians form wild systems that are much more complicated for the investigation
\cite{Davies, Krejcirik, Mit}.

Another approach to the generalization of Riesz bases is based on the rigged Hilbert spaces framework instead of the original Hilbert space
\cite{Trapani}.

In the present paper, we  study generalized Riesz systems (GRS) which was originally introduced in \cite{Inoue3, Inoue} and then,
slightly modified in \cite{BIT, Kuzhel}. In order to explain the idea of definition we note that vectors of a Riesz basis $\{\phi_n\}$  
have the form $\phi_n=Re_n$, where
$R$ is  a bounded and boundedly invertible operator in a Hilbert space $\cH$ and $\{e_n\}$ is an orthonormal basis (ONB) of $\cH$.
Using the polar decomposition of $R=|R^*|U=e^{Q/2}U$, where $U$ is a unitary operator in $\cH$, we arrive at the conclusion:
 \emph{a sequence $\{\phi_n\}$ is called a Riesz basis if there exists a bounded self-adjoint operator $Q$ in $\cH$ and 
an ONB $\{e_n\}$  such that  $\phi_n=e^{Q/2}e_n$.}
This simple observation leads to: 

\begin{ddd}\label{d1}
A sequence $\{\phi_n\}$ is called a generalized Riesz system (GRS) 
if there exists a self-adjoint operator $Q$ in $\cH$ and an ONB $\{e_n\}$  such that  $e_n\in\mathcal{D}(e^{Q/2})\cap\mathcal{D}(e^{-Q/2})$ and 
$\phi_n=e^{Q/2}e_n.$	
\end{ddd}

For a GRS  $\{\phi_n\}$, the dual  GRS is determined by the formula $\{\psi_n=e^{-Q/2}e_n\}$.
Obviously, $\{\phi_n\}$ and  $\{\psi_n\}$ are bi-orthogonal sequences.

Dual GRS's can be considered as a particular case of 
$\mathcal{G}$-quasi bases introduced by Bagarello in \cite{bag2013JMP},  
and then analyzed in a series of papers, see, \cite{BGSZ} and the references therein. 

The main objective of the paper is  to further development of the GRS theory. 
In contrast to the standard approach \cite{BB, BIT,  Inoue1, Inoue3,  Inoue4, Inoue}, where GRS were mainly used as auxiliary tools  for the definition and investigation of
manifestly non self-adjoint Hamiltonians  and relevant physical operators, we consider GRS as a self-contained object of the basis theory \cite{Christ, Heil}.
Our studies are based on advanced methods of extension theory of symmetric operators, see \cite{AK_Arlin, AT} and Subsection \ref{sec2.1}.
% and we believe that 
%such kind of synergy allows one to have the capability to provide some deeper insights into the structural subtleties of GRS-theory.

We say that a sequence of vectors $\{\phi_n\}$ of a Hilbert space $\cH$ is 
\emph{semi-regular} if  $\{\phi_n\}$ is minimal and complete in $\cH$. 
The minimality of $\{\phi_n\}$ yields the existence of a bi-orthogonal sequence $\{\psi_n\}$, while 
the completeness of $\{\phi_n\}$  guarantees the uniqueness of  $\{\psi_n\}$.  
The positive symmetric operator $S$ mapping $\{\phi_n\}$ onto $\{\psi_n\}$, see \eqref{new14} plays an important role in our studies.
We show that a semi-regular sequence $\{\phi_n\}$  is a GRS if and only if the Friedrichs extension $A_F$ of $S$ is a positive 
operator (Theorem \ref{new31}).   Another criterion of being GRS was established in \cite[Theorem 3.4]{Inoue3}
by methods based on the investigation of special operators associated with $\{\phi_n\}$ and $\{e_n\}$.

%In \cite{Inoue3}, semi-regular sequences were defined in more restrictive way assuming additionally that
%the set $D(\phi)=\{f\in\cH : \sum|(f, \phi_n)|^2<\infty\}$ is dense in $\cH$.  
%It is worth mentioning that the items (2) and (3) of \cite[Theorem 3.4]{Inoue3}

Theorem \ref{new31} allows one to explain the phenomenon of nonuniqueness of self-adjoint operators $Q$
in the definition of GRS (Subsection \ref{subsec2.3}). Further we show that each semi-regular sequence $\{\phi_n\}$
with the property of being Bessel sequence  has to be a GRS and we characterize this important case in terms of $Q$ 
(Theorem \ref{KKK4}). The  Olevskii's result \cite[Theorem 1]{Olew} allows one to establish a relationship between  essential
spectra of  self-adjoint operators $Q$ and conditional bounded bases (Proposition \ref{ppp21}). 
 
At the end of Section \ref{sec2}  we concentrate on an important particular case 
(which fits well the specific of $\mathcal{PT}$-symmetric Hamiltonians) where a semi-regular sequence $\{\phi_n\}$ is
 $J$-orthonormal and we combine the GRS-related approach with Krein spaces based methods \cite[Chap. 6 ]{BGSZ}, \cite{KKS}.
 Following \cite{Kuzhel}, we define $J$-orthonormal sequences of the first/second type and discuss advantages of 
 first type sequences.  In particular, eigenstates of the shifted harmonic oscillator form a $J$-orthonormal sequence of the first
 type and it seems natural to suppose that eigenstates of a $\mathcal{PT}$-symmetric Hamiltonian with 
 unbroken $\mathcal{PT}$-symmetry \cite[p. 41]{Bender} form a first type sequence.   In Subsection \ref{sec3.2} 
 we present a general method which allows one to construct of the first/second type sequences.
 
Throughout the paper, $\mathcal{D}(A)$, $\mathcal{R}(A)$, and 
$\ker{A}$ denote the domain, the range, and the null-space of a linear operator
$A$, respectively. The symbol $\{v_n\}$ means the collection of vectors $v_n$
parametrized by a set $\mathcal{I}$ of integers. Usually, $\mathcal{I}=\mathbb{N}$.   

\section{General theory of GRS}\label{sec2}

\subsection{Preliminaries}\label{sec2.1}
Here all necessary results of extension theory of symmetric operators are presented in a form convenient for our exposition.
The  articles \cite{Ando, AK_Arlin, AT} and \cite[Chap. 10]{Konrad} are recommended as complementary reading on the subject.

Let $\mathcal{H}$ be a complex Hilbert space with inner product  $(\cdot, \cdot)$ linear in the first argument.
 An operator $A$ is called  positive [nonnegative] if   $(Af,f)>0$  \  [$(Af,f)\geq{0}$] for non-zero $f\in\mathcal{D}(A)$. 

Let  $A$ and $B$ be nonnegative self-adjoint operators. We say that $A$ \emph{is greater or equal} $B$, i.e., $A\geq{B}$
if
\begin{equation}\label{AAA1}
 \mathcal{D}(A^{1/2})\subseteq\mathcal{D}(B^{1/2})  \quad \mbox{and} \quad \|A^{1/2}f\|\geq\|B^{1/2}f\|, \quad  f\in\mathcal{D}(A^{1/2}).
 \end{equation}

The next technical result follows  from \eqref{AAA1} (see \cite[Corollary 10.12]{Konrad}). 
 \begin{lelele}\label{AAA2}
 If $A\geq{B}$ and $B$ is positive, then $A$ is also positive.  
\end{lelele}

Let $S$ be a nonnegative densely defined operator in $\cH$.  M. Krein established that the set of nonnegative self-adjoint
extensions $\{A\}$ of $S$ can be ordered as follows \cite[Theorem 3.5]{AT}: $A_F\geq{A}\geq{A_K}$,
where the greatest self-adjoint extension $A_F$ is called the Friedrichs extension, while the smallest one $A_K$
is called the Krein-von Neumann extension. 

The extensions $A_K$ and $A_F$ are examples of extremal extensions.  
We recall \cite{AK_Arlin}  that a nonnegative self-adjoint extension  $A$ of
$S$ is called \emph{extremal} if
\begin{equation}\label{bebe86}
\inf_{f\in\mathcal{D}(S)}{(A(\phi-f),(\phi-f))}=0 \quad \mbox{for all} \quad \phi\in\mathcal{D}(A).
\end{equation}

If a nonnegative self-adjoint extension $A$ of $S$ is positive, we can set $A=e^{-Q}$, where $Q=-\ln{A}$ is a self-adjoint operator in $\cH$ and 
define the new Hilbert space $\cH_{-Q}$ as the completion of  $\mathcal{D}(A)=\mathcal{D}(e^{-Q})$ 
with respect to the new inner product
\begin{equation}\label{new1}
(f, g)_{-Q}:=(e^{-Q}f, g)=(e^{-Q/2}f, e^{-Q/2}g),  \qquad f,g\in\mathcal{D}(e^{-Q}).
\end{equation}
 Rewriting \eqref{bebe86} as
$$
\inf_{f\in\mathcal{D}(S)}{(\phi-f, \phi-f)_{-Q}}=\inf_{f\in\mathcal{D}(S)}\|\phi-f\|_{-Q}^2=0 \quad \mbox{for all} \quad \phi\in\mathcal{D}(e^{-Q})
$$  
we obtain
 \begin{lelele}\label{AAA3}
 A positive self-adjoint extension $A=e^{-Q}$ of $S$ is an extremal extension if and only if 
 $\mathcal{D}(S)$ is a dense set in $\cH_{-Q}$.
  \end{lelele}

A non-negative \emph{non-densely defined} symmetric operator $S$ admits self-adjoint extensions, but not necessarily non-negative ones.
The well-known Ando-Nishio result is \cite[Theorem 1]{Ando}:
 \begin{lelele}\label{AAA7}
 A closed non-negative symmetric operator $S$ admits a non-negative self-adjoint extension
 if and only if it is positively closable, i.e., if the relations
 $$
 \lim_{n\to\infty}(Sf_n, f_n)=0 \qquad \mbox{and} 
 \qquad  \lim_{n\to\infty}Sf_n=g
$$
implies $g=0$. 
\end{lelele}

Lemma \ref{AAA7} is evident for densely defined operators because 
each densely defined operator $S$ is positively closable \cite[p. 67]{Ando}.
Another useful result follows from \cite[Corollary 4]{Ando}:
 \begin{lelele}\label{AAA8}
 Let $S$ be a closed densely defined positive operator. Then $S^{-1}$ admits a non-negative 
 self-adjoint extension if and only if the Friedrich extension $A_F$ of $S$ is positive.
 \end{lelele}

\subsection{Conditions of being GRS}
 Let $\{\phi_n\}$ be a GRS. In view of Definition \ref{d1}, 
the sequence  $\{\psi_n=e^{-Q/2}e_n\}$
is well defined and it is a bi-orthogonal sequence for $\{\phi_n=e^{Q/2}e_n\}$.
 Obviously, $\{\psi_n\}$ is a GRS which we call \emph{a dual GRS}.

The existence of a bi-orthogonal sequence means that each GRS $\{\phi_n\}$ has to be 
a \textit{minimal sequence}, i.e.,   $\phi_j\notin\overline{span}{\{\phi_k\}_{k\not=j}}$ \cite[Lemma 3.3.1]{Christ}.
However, not each minimal sequence is a GRS.  

We say that a minimal sequence $\{\phi_n\}$  is 
\emph{semi-regular} if  $\{\phi_n\}$  is complete in $\cH$ and 
\emph{regular} if its  bi-orthogonal sequence $\{\psi_n\}$ is also complete.
For a semi-regular sequence $\{\phi_n\}$ the corresponding bi-orthogonal sequence $\{\psi_n\}$ is
determined uniquely.

Let $\{\phi_n\}$ be a minimal sequence. Then there exists a bi-orthogonal sequence  $\{\psi_n\}$ and 
we can consider an  operator $S$ defined initially on  
	\begin{equation}\label{new14}
	S\phi_n=\psi_n 
	\end{equation}
and extended on ${\mathcal D}(S)=span\{\phi_n\}$  by the linearity. 
By the construction, 
	$$ 
	(Sf, f)=\sum_{n=1}^k\sum_{m=1}^kc_n\overline{c}_m(\psi_n, \phi_m)=\sum_{n=1}^k{|c_n|^2} \quad \mbox{for all} \quad f=\sum_{n=1}^k{c_n}\phi_n\in{\mathcal{D}(S)}.
	$$ 
Therefore, $S$ is a positive operator.  For a semi-regular sequence $\{\phi_n\}$,  the operator $S$ is densely defined and the 
Friedrichs extension $A_F$ of $S$ exists.
 
\begin{ttt}\label{new31}  	 	
Let $\{\phi_n\}$ be a semi-regular sequence. The following are equivalent:
\begin{enumerate}
\item[(i)] $\{\phi_n\}$ is a GRS;
\item[(ii)] the Friedrichs extension $A_F$ of $S$ is a positive operator;
\item[(iii)] the closure $\overline{S}$ of $S$ is a positive operator and  the relations
\begin{equation}\label{KKK1}
\lim_{n\to\infty}(\overline{S}f_n, f_n)=0 \qquad \mbox{and} 
 \qquad  \lim_{n\to\infty}f_n=g
\end{equation}
imply that $g=0$.
\end{enumerate}
\end{ttt}
\begin{proof} $(i)\to(ii)$.
If $\{\phi_n\}$ is a GRS, then
$e^{-Q}\phi_n=e^{-Q/2}e_n=\psi_n$.  In view of \eqref{new14},  
  $e^{-Q}$ is  a positive self-adjoint extension of $S$ and
 $A_F\geq{e^{-Q}}$ since the Friedrichs extension is the greatest nonnegative
self-adjoint extension of $S$. By virtue of  Lemma \ref{AAA2},  $A_F$ is positive. 

$(ii)\to(i)$. The positivity of $A_F$  means that $A_F=e^{-Q}$,  where $Q$ is a self-adjoint operator in $\cH$. 
Denote $e_n=e^{-Q/2}\phi_n$. Due to \eqref{new14}, $e_n=e^{Q/2}\psi_n$. Therefore,  
	$e_n\in\mathcal{D}(e^{Q/2})\cap\mathcal{D}(e^{-Q/2})$  and $(e_n, e_m)=(e^{-Q/2}\phi_n,  e^{Q/2}\psi_m)=(\phi_n, \psi_m)=\delta_{nm}$.

The orthonormal sequence  $\{e_n\}$  turns out to be an ONB if  $\{e_n\}$ is complete in $\cH$.
Assume that $\gamma$ is orthogonal to $\{e_n\}$ in $\cH$. Then there exists 
a sequence $\{f_m\}$  $(f_m\in\mathcal{D}(e^{-Q}))$ such that  $e^{-Q/2}{f_m}\to\gamma$ in $\cH$ 
(because $e^{-Q/2}\mathcal{D}(e^{-Q})$ is a dense set in $\cH$). 
 In this case, due to \eqref{new1},  $\{f_m\}$ is a Cauchy sequence in $\cH_{-Q}$ and therefore,  ${f}_m$ tends to some $f\in\cH_{-Q}$. 
This means that	
    \begin{equation}\label{neww4}
	0=(\gamma,  e_n)=\lim_{m\to\infty}(e^{-Q/2}{f_m}, e_n)=\lim_{m\to\infty}(f_m, \phi_n)_{-Q}=(f, \phi_n)_{-Q}.
	\end{equation}

By Lemma \ref{AAA3},  the set $\mathcal{D}(S)=span\{\phi_n\}$ is dense in the Hilbert space $\cH_{-Q}$. 
 In view of \eqref{neww4}, $f=0$ that means
$\lim_{m\to\infty}\|f_m\|_{-Q}=\lim_{m\to\infty}\|e^{-Q/2}f_m\|=0$ and therefore,  $\gamma=0$.
		 Hence,  $\{e_n\}$ is complete in $\cH$ and $\{e_n\}$ is an ONB of $\cH$.
 		 
The implication $(ii)\to(iii)$ is obvious.

$(iii)\to(ii)$.  The operator $\overline{S}$ has the inverse ${\overline{S}}^{-1}$ since $\overline{S}$ is positive.
The operator ${\overline{S}}^{-1}$ is closed and \eqref{KKK1} can be rewritten as follows:
$$
\lim_{n\to\infty}(g_n, {\overline{S}}^{-1}g_n)=0 \qquad \mbox{and}  \qquad  \lim_{n\to\infty}{\overline{S}}^{-1}g_n=g \qquad (g_n=\overline{S}f_n) 
$$ 		 
This means that ${\overline{S}}^{-1}$ is positively closable. Hence, it admits a non-negative self-adjoint extension (Lemma \ref{AAA7}) 
 Applying Lemma \ref{AAA8} we complete the proof.
\end{proof}
\begin{ccc}\label{KKK2}
A regular sequence is a GRS. 
	\end{ccc}
\begin{proof}
Let $\{\phi_n\}$ be a regular sequence. Its regularity means that $\mathcal{R}(S)$ is a dense set in $\cH$.
The latter means that each nonnegative self-adjoint extension $A$ of $S$ must be positive.  In particular,
the Friedrichs extension  $A_F$ is  positive. By Theorem \ref{new31},  $\{\phi_n\}$
is a GRS. 
 \end{proof}	

A shifting of the orthonormal Hermite functions  $e_n(x)$ in the complex plane gives rise to regular sequences in $L_2(\mathbb{R})$.
 In particular, eigenfunctions of the shifted harmonic oscillator $\{\phi_n(x)=e_n(x+ia)\}$ form a regular sequence and $\phi_n=e^{Q/2}e_n$, where
 $Q=2ai\frac{d}{dx}$ is an unbounded self-adjoint operator in $L_2(\mathbb{R})$  \cite[Subsection IV.1]{Kuzhel}.  We refer \cite{Inoue1, Inoue} for 
 the relationship between general regular sequences and some physical operators.
   
\begin{rem}
Theorem \ref{new31} can be  generalized to the case of non-complete minimal sequence $\{\phi_n\}$ such that
its bi-orthogonal $\{\psi_n\}$ is also non-complete. 
In this case, $S$ is a non-densely defined positive symmetric operator in $\cH$. We should suppose the existence
of a positive self-adjoint extension $A$ of $S$.
Similarly to the proof of Theorem  \ref{new31} we set $A=e^{-Q}$ and determine orthonormal system $\{e_n\}$ in $\cH$.
By virtue of \eqref{neww4},  the completeness of $\{e_n\}$ in $\cH$ is equivalent to the completeness of $\{\phi_n\}$ in $\cH_{-Q}$. 
This means that $\{e_n\}$ is an ONB in $\cH$ if $\{\phi_n\}$ is a complete set in $\cH_{-Q}$.
Summing up: \emph{Let $\{\phi_n\}$ be a minimal sequence and let $\{\psi_n\}$ be its bi-orthogonal sequence. These sequences  are GRS if and only if
 there exists a positive self-adjoint extension $A=e^{-Q}$ of $S$ such that $\{\phi_n\}$ is a complete set in the Hilbert space  $\cH_{-Q}$.}
 
 Another approach to the study of non-complete minimal sequences can be found in \cite{BIT}.
\end{rem}

\subsection{The uniqueness of $Q$ in the definition of GRS}\label{subsec2.3}

Let $\{\phi_n\}$ be a basis in $\cH$.  Then $\{\phi_n\}$ is a regular sequence because
its bi-orthogonal sequence $\{\psi_n\}$ has to be a basis \cite[Corollary 5.22]{Heil}.  By Corollary \ref{KKK2}, 
$\{\phi_n\}$ is a GRS, i.e., $\phi_n=e^{Q/2}e_n$.  Moreover,
by \cite[Proposition II.9]{Kuzhel},  
the pair $(Q, \{e_n\})$ in Definition \ref{d1} is determined \emph{uniquely} for every basis $\{\phi_n\}$.
For this reason,  a natural question arise: \emph{is the pair $(Q, \{e_n\})$ determined uniquely for a  given GRS $\{\phi_n\}$?} 

The choice of the Friedrichs extension $A_F=e^{-Q}$ of $S$ in the proof of Theorem \ref{new31} was related to the fact 
that  $\mathcal{D}(S)$ must be dense in the Hilbert space $\cH_{-Q}$ 
(that, in view of \eqref{neww4}, is equivalent to the completeness of $\{e_n\}$ in $\cH$). 
Due to Lemma \ref{AAA3}, each positive extremal extension $A=e^{-Q}$ can be used instead of $A_F$ in the proof of Theorem \ref{new31}.
This observation leads to the following result (\cite[Proposition II.10]{Kuzhel}):

\begin{ppp}\label{KKK5}
Let a semi-regular sequence $\{\phi_n\}$ be  a GRS. Then 
 a self-adjoint operator $Q$ and an ONB  $\{e_n\}$  are determined uniquely in the formula $\phi_n=e^{Q/2}e_n$ 
 if and only the symmetric operator $S$ in \eqref{new14} has a unique positive extremal extension.
\end{ppp}

\begin{rem}
The above mentioned unique positive extremal extension coincides with the Friedrichs extension
$A_F$. Indeed, the existence of another positive extension $A\not=A_F$ means
that $A_F$ is also positive (Lemma \ref{AAA2}). Due to the uniqueness of positive extension, we get $A=A_F$.
\end{rem}

\subsection{Bases and Bessel sequences}
 Various classes of GRS's can be easily characterized in terms of  spectral properties of the corresponding self-adjoint  
operators $Q$. 
 
We recall that a semi-regular sequence $\{\phi_n\}$ is called  \emph{a Riesz basis} if there exists $0<a\leq{b}$ such that\footnote{we refer \cite[Theorems 7.13, 8.32]{Heil} 
for equivalent definitions of Riesz basis} 
\begin{equation}\label{KK1}
{a}\|f\|^2\leq\sum_{n}|(f, \phi_n)|^2\leq{b}\|f\|^2  \quad  \mbox{for all} \quad  f\in\mathcal{H}.
\end{equation}

\begin{ppp}\label{KKK5b}
The following are equivalent:
\begin{enumerate}
\item[(i)] a sequence $\{\phi_n\}$ is a Riesz basis with bounds $0<a\leq{b}$;
\item[(ii)] $\{\phi_n\}$ is a GRS, i.e.,  $\phi_n=e^{Q/2}e_n$, where $Q$ is bounded self-adjoint operator such
that  $\sigma(Q)\subset[\ln{a}, \ln{b}]$.
\end{enumerate}
\end{ppp}
\begin{proof} 
If $\{\phi_n\}$ is a Riesz basis, then $\phi_n=e^{Q/2}e_n$, where $Q$ is a bounded operator (see Section \ref{sec1}).
The substitution of $\phi_n=e^{Q/2}e_n$ into \eqref{KK1}  gives
$$
a(f,f)={a}\|f\|^2\leq\sum_{n}|(e^{Q/2}f, e_n)|^2\leq\|e^{Q/2}f\|^2=(e^Qf, f)\leq{b}\|f\|^2=b(f,f).
$$
Therefore, $aI\leq{e^Q}\leq{bI}$ that justifies $(i)\to(ii)$. The converse statement is obvious.
\end{proof}

\begin{lelele}\label{AAA11}
Let $Q$ be a self-adjoint operator such that $\sigma(Q)\subset(-\infty, \ln{b}]$ and let $\{e_n\}$ be an arbitrary 
ONB of $\cH$. Then the sequence $\{\phi_n=e^{Q/2}e_n\}$ is a GRS with the pair $(Q, \{e_n\})$.
\end{lelele}
\begin{proof}
Since $\sigma(Q)\subset(-\infty, \ln{b}]$, the self-adjoint operator $e^{Q/2}$ is  bounded.
Hence, the elements $\phi_n=e^{Q/2}e_n$ are well-defined.  If $\gamma$ is orthogonal to $\{\phi_n\}$, then
$0=(\gamma, e^{Q/2}e_n)=(e^{Q/2}\gamma, e_n)$ implies that  $e^{Q/2}\gamma=\gamma=0$. Therefore
 $\{\phi_n\}$ is semi-regular and  its bi-orthogonal sequence $\{\psi_n\}$ is defined uniquely.
According to Definition \ref{d1}, it is  sufficient to show that $e_n\in\mathcal{D}(e^{-Q/2})$ and $\psi_n=e^{-Q/2}e_n$.
By virtue of the relation
 $\delta_{mn}=(\phi_m, \psi_n)=(e^{Q/2}e_m, \psi_n)=(e_m, e^{Q/2}\psi_n)$ we obtain that
 $e^{Q/2}\psi_n=e_n$ . The last relation means that  $e_n\in\mathcal{D}(e^{-Q/2})$ and $\psi_n=e^{-Q/2}e_n$.
\end{proof}

A  sequence  $\{\phi_n\}$ is called \emph{a Bessel sequence} if there exists $b>0$ such that
\begin{equation}\label{e2}
\sum_{n}|(f, \phi_n)|^2\leq{b}\|f\|^2 \quad  \mbox{for all} \quad  f\in\mathcal{H}.
\end{equation}

\begin{ttt}\label{KKK4}
 The following are equivalent:
\begin{enumerate}
\item[(i)]  a semi-regular sequence $\{\phi_n\}$ is a Bessel sequence; 
\item[(ii)] $\{\phi_n\}$ is a GRS, i.e.,  $\phi_n=e^{Q/2}e_n$, where $Q$ is a self-adjoint operator such that $\sigma(Q)\subset(-\infty, \ln{b}]$.
\end{enumerate}
\end{ttt}
\begin{proof}
 $(i)\to(ii)$. If $\{\phi_n\}$ is a Bessel sequence, then the synthesis operator $R\{c_n\}=\sum{c_n\phi_n}$ defines a bounded operator which
 maps $l^2(\mathbb{N})$ into $\cH$ \cite[p. 190]{Heil}. The minimality of
 $\{\phi_n\}$ implies that $\{\phi_n\}$ is $\omega$-independent \cite[p. 156]{Heil}. 
The latter means that the series $\sum{c_n\phi_n}$ converges and equal $0$ only when $c_n=0$.
Therefore, $\ker{R}=\{0\}$. 
 
 Let $\{\delta_n\}$ be the canonical basis of $l^2(\mathbb{N})$.  Then 
 $R\delta_n=\phi_n$. Identifying  $\{\delta_n\}$ with an ONB $\{\tilde{e}_n\}$ of  $\cH$ we obtain a bounded operator\footnote{we keep the same notation $R$ for the operator in $\cH$.} 
 $R$ in $\cH$ such
 $R\tilde{e}_n=\phi_n$. The polar decomposition of $R$ is  $R=|R^*|U$, where $|R^*|=\sqrt{RR^*}$ and
 $U$ is an isometric operator mapping the closure of  ${\mathcal R}(\sqrt{R^*R})$ onto the closure of ${\mathcal R}(R)$ \cite[Chapter VI, Subsect. 2.7]{Kato}.
 We remark that ${\mathcal R}(\sqrt{R^*R})$ and ${\mathcal R}(R)$ are dense sets in $\cH$, 
 since, respectively,  $\ker\sqrt{R^*R}=\ker{R}=\{0\}$ and $\{\phi_n\}$ is a complete set in $\cH$. 
 Therefore, $U$ is a unitary operator in $\cH$. Moreover,
  $|R^*|$ is a positive bounded self-adjoint operator (since $\ker|R^*|=\ker{R}=\{0\}$). 
 
 The positivity of $|R^*|$ leads to the formula $|R^*|=e^{Q/2}$, where $Q$ is a self-adjoint operator in $\cH$. 
  Denote $e_n=U\tilde{e}_n$. Obviously, $\{e_n\}$ is an ONB of $\cH$ and
  $$
  \phi_n=R\tilde{e}_n=|R^*|U\tilde{e}_n=e^{Q/2}U\tilde{e}_n=e^{Q/2}e_n.
  $$
  After the substitution of  $\phi_n=e^{Q/2}e_n$ into \eqref{e2}:
 $$
  \sum_{n}|(f, e^{Q/2}e_n|^2=\sum_{n}|(e^{Q/2}f, e_n|^2=\|e^{Q/2}f\|^2=(e^{Q}f, f)\leq{b}(f, f). 
$$ 
 The obtained inequality leads to the conclusion that $\sigma(Q)\subset(-\infty, \ln{b}]$. 
 Applying  Lemma \ref{AAA11} we complete  the proof  of  $(i)\to(ii)$. 
 
$(ii)\to(i)$.  In view of Lemma \ref{AAA11},  $\{\phi_n\}$ is a semi-regular sequence. 
The operator $e^Q$ is bounded and $\|e^Q\|\leq{b}$  (since  $\sigma(Q)\subset(-\infty, \ln{b}]$). Hence,
$$
\sum_{n}|(f, \phi_n)|^2= \sum_{n}|(e^{Q/2}f, e_n|^2=\|e^{Q/2}f\|^2=(e^{Q}f, f)\leq{b}\|f\|^2  
$$
that completes the proof.
\end{proof}

A sequence $\{\phi_n\}$ is called \emph{bounded} if $0<a\leq\|\phi_n\|\leq{b}$ for all $n$.
A basis  $\{\phi_n\}$  is called  \emph{conditional} if its property of being basis depends on the permutation of elements $\phi_n$.
 
\begin{ppp}\label{ppp21}
Let $Q$ be a self-adjoint operator in $\cH$ such that $\sigma(Q)\subset(-\infty, \ln{b}]$.
 The following are equivalent:
\begin{enumerate}
\item[(i)]  there exists an ONB $\{e_n\}$ of $\cH$ such that the sequence $\{\phi_n=e^{Q/2}e_n\}$ is a conditional bounded basis; 
\item[(ii)] there exists $\beta<0$  such that each interval $[(n+1)\beta, n\beta]$  ($n=0,1,\ldots$) includes at least one point of essential spectrum 
of $Q$.
\end{enumerate}
\end{ppp}
\begin{proof} Applying \cite[Theorem 1]{Olew} to the positive bounded operator $e^{Q/2}$ and taking
into account properties of an essential spectrum \cite[Proposition 8.11]{Konrad}  we arrive at the conclusion that the item $(i)$  
is equivalent to the existence of $0<q<1$ such that the essential spectrum of $e^{Q/2}$ has a non-zero interaction with each
interval $[q^{n+1}, q^n]$. Since $Q=2\ln{e^{Q/2}}$, the later statement is equivalent to $(ii)$ with $\beta=2\ln{q}$. 
\end{proof}

Let $\cH=L_2(-\pi, \pi)$ and $Q$ is an  operator of multiplication by $\alpha\ln|x|$\  ($0<\alpha$) in $\cH$. 
Obviously, $Q$ is self-adjoint,  its spectrum coincides with $(-\infty, \ln{\pi^\alpha}]$ and it is essential.
By Proposition \ref{ppp21},  there exists an ONB $\{e_n\}$ of $L_2(-\pi, \pi)$ such that $\{\phi_n=e^{Q/2}e_n\}$ is a conditional bounded basis.
In view of the Babenko example \cite[Example 5.13]{Heil},  for $0<\alpha<\frac{1}{2}$, the corresponding ONB can be chosen as 
$\{e_n=\frac{1}{\sqrt{2\pi}}e^{inx}\}_{-\infty}^\infty$.

\subsection{$J$-orthonormal sequences  and GRS}
 Let $J$ be a bounded self-adjoint operator in a Hilbert space  $\cH$ such that $J^2=I$.   The Hilbert space
$\cH$ equipped with the indefinite inner product
$[\cdot, \cdot]:=(J\cdot, \cdot)$  is called a \emph{Krein space}. 

A sequence $\{\phi_n\}$ is called \emph{$J$-orthonormal} if  $|[\phi_n, \phi_m]|=\delta_{nm}$.  

Each  $J$-orthonormal sequence $\{\phi_n\}$ is minimal  since its  
bi-orthogonal one is determined as
 \begin{equation}\label{new4}
	\psi_n=[\phi_n,\phi_n]J\phi_n.
	\end{equation}
 In view of \eqref{new4}, the positive symmetric operator $S$ in \eqref{new14} acts as
 $S\phi_n=[\phi_n, \phi_n]J\phi_n$.  
\begin{ppp}
Let $\{\phi_n\}$ be a complete $J$-orthonormal sequence. Then
$\{\phi_n\}$ is a Bessel sequence if and only if $\{\phi_n\}$  is a Riesz basis.
\end{ppp}
\begin{proof}
 Let us assume that  $\{\phi_n\}$ 
is a Bessel sequence. 	 Then $\{\psi_n\}$  is also a Bessel sequence. 
Indeed, substituting $Jf$ instead of $f$  into 
\eqref{e2} and using \eqref{new4}, we obtain
$$
\sum_{n}|(Jf, \phi_n)|^2=\sum_{n}|(f, J\phi_n)|^2=\sum_{n}|(f, \psi_n)|^2\leq{b}\|Jf\|^2=b\|f\|^2.
$$
 
By Theorem \ref{KKK4},  $\{\phi_n\}$ is a GRS and  $\phi_n=e^{Q/2}e_n$, where $\sigma(Q)\subset(-\infty, \ln{b}]$.
Since $\{\psi_n\}$ is also a Bessel sequence, applying Theorem \ref{KKK4} again we obtain
 $\sigma(-Q)\subset(-\infty, \ln{b}]$ or $\sigma(Q)\subset[-\ln{b}, \infty)$.
Therefore,  $\sigma(Q)\subset[\ln{a}, \ln{b}]$,
 where $a=1/b$.  In view of Proposition \ref{KKK5b},  $\{\phi_n\}$ is a Riesz basis. The inverse statement is obvious.  
\end{proof}

If $\{\phi_n\}$ is complete in $\cH$, then $\{\psi_n\}$ in \eqref{new4} is complete too. 
Therefore, $\{\phi_n\}$  is regular and, by Corollary \ref{KKK2},  $\{\phi_n\}$ is a GRS. Thus,
\emph{each complete $J$-orthonormal sequence is a GRS}.
 
It follows from the proof of Corollary \ref{KKK2} that each extremal extension $A$ of $S$ is positve  
Therefore, the corresponding operator $Q=-\ln A$ in Definition \ref{d1} can be determined by every extremal extension $A$.
If  $Q$ is determined uniquely, then \cite[Theorem III.3]{Kuzhel}:
\begin{equation}\label{new34}
	JQ=-QJ.
\end{equation}
However, if  $Q$ is not determined uniquely,  not each $Q=-\ln{A}$ satisfies \eqref{new34}. 
 In particular, as follows from \cite{KKS},  the operator $Q$ that corresponds 
to the Friedrichs extension $A_F$ does not satisfy \eqref{new34}. Moreover, there exist
 complete $J$-orthonormal sequences for which no operators $Q$ satisfying \eqref{new34} can be found. 

We say that \emph{a complete $J$-orthonormal sequence $\{\phi_n\}$ is of the first type} if
there exists a self-adjoint operator $Q$ in Definition \ref{d1} such that  \eqref{new34} holds. 
 Otherwise, $\{\phi_n\}$ is of \emph{the second type}.

 $J$-orthonormal bases are examples of the first type sequences.
The next  statement was proved in \cite{Kuzhel}, where  the notation "quasi-bases"  was used for the first type sequences.  
\begin{ppp}\label{new2b} 
The following are equivalent:
\begin{enumerate}
\item[(i)]  a complete $J$-orthonormal sequence $\{\phi_n\}$ is of the first type; 
\item[(ii)] the sequence  $\{\phi_n\}$ is regular and the corresponding pair $(Q, \{e_n\})$ in Definition \ref{d1} can be chosen as follows:
$Q$ satisfies \eqref{new34} and $e_n$ are eigenfunctions of $J$, i.e., $Je_n=e_n$ or $Je_n=-e_n$. 
\end{enumerate}
\end{ppp}

In what follows, considering a first type sequence $\{\phi_n=e^{Q/2}e_n\}$, we assume that the pair $(Q, \{e_n\})$ satisfies conditions $(ii)$ of Proposition \ref{new2b}.
A detailed analysis of the first/second type sequences can be found in \cite{Kuzhel}.  We just mention that
a first type sequence $\{\psi_n=e^{Q/2}e_n\}$ generates a $\cC$-symmetry operator\footnote{the concept of $\cC$-symmetry 
is widely used in $\cP\cT$-symmetric quantum mechanics \cite{BGSZ, Bender}} $\cC=e^QJ$ with the
 same operator $Q$.  The latter allows one to construct the new Hilbert space  $\cH_{-Q}$
involving $\{\phi_n\}$ as ONB,  directly as the completion of $\mathcal{D}(\cC)$ with respect to ``$\cC\mathcal{PT}$-norm'': 
 $(\cdot, \cdot)_{-Q}=[\cC\cdot, \cdot]=(Je^QJ\cdot,\cdot)=(e^{-Q}\cdot, \cdot).$

For a second type sequence,  the inner product  $(\cdot,\cdot)_{-Q}$ defined by \eqref{new1} cannot be expressed via  $[\cdot, \cdot]$
 and one should apply much more efforts for the precise definition of $(\cdot,\cdot)_{-Q}$.

\section{Examples}\label{sec3}
\subsection{A semi-regular sequence which cannot be a GRS}\label{subsec3.1}
Let $\{\mathsf{e}_n\}_{n=0}^{\infty}$ be an ONB of $\cH$.
Denote 
$$
\phi_n=\frac{1}{n^\beta}\mathsf{e}_n+\frac{1}{n^\alpha}\mathsf{e}_0, \qquad \alpha, \beta\in\mathbb{R}, \quad n=1,2,\ldots
$$
The sequence  $\{\phi_n\}_{n=1}^\infty$  is minimal  since $\{\psi_n=n^{\beta}\mathsf{e}_n\}_{n=1}^\infty$
is bi-orthogonal to $\{\phi_n\}$.  It is easy to see that $\{\phi_n\}$ is complete in $\cH$ if and only if 
$\alpha-\beta\leq\frac{1}{2}$. The last relation determines admissible parameters $\alpha, \beta$ for which
$\{\phi_n\}$ is a semi-regular sequence. During the  subsection we suppose that this inequality holds.

In view of  \eqref{new14},  $S(\frac{1}{n^\beta}\mathsf{e}_n+\frac{1}{n^\alpha}\mathsf{e}_0)=n^{\beta}\mathsf{e}_n$ and the operator 
$S$ can be described as:
\begin{equation}\label{AAA1b} 
Sf=\sum_{n=1}^k{n^{2\beta}}c_n\mathsf{e}_n  \quad \mbox{for all} \quad f=\sum_{n=1}^k{c_n}\mathsf{e}_n+\left(\sum_{n=1}^k\frac{n^{\beta}c_n}{n^{\alpha}}\right)\mathsf{e}_0 \ \in \mathcal{D}(S).
\end{equation}
It follows from \eqref{AAA1b} that the non-negative self-adjoint operator
\begin{equation}\label{AAA2b} 
Af=A\left(\sum_{n=0}^\infty{c}_n\mathsf{e}_n\right)=\sum_{n=1}^\infty{n^{2\beta}}c_n\mathsf{e}_n  
\end{equation}
with the domain 
$\mathcal{D}(A)=\left\{ f=\sum_{n=0}^\infty{c_n}\mathsf{e}_n  :   \{c_n\}_{n=1}^\infty,  \{n^{2\beta}c_n\}_{n=1}^\infty\in{\ell_2(\mathbb{N})}\right\}$
is an extension of $S$.

Assume that $\beta\leq{0}$. Then the semi-regular sequence $\{\phi_n\}$ cannot be a GRS.
Indeed, in this case, the operator $A$ is bounded. Therefore, $A$  coincides with the closure $\overline{S}$ of $S$.
In view of \eqref{AAA2b},  $\overline{S}\mathsf{e}_0=A\mathsf{e}_0=0$. By Theorem \ref{new31}, $\{\phi_n\}$ cannot be a GRS. 

Assume now that $\beta>0$. Then  $A$ is an unbounded non-negative
self-adjoint extension of $S$. Hence, $A$ is an extension of $\overline{S}$.  Using \eqref{AAA1b} and
\eqref{AAA2b}, we obtain
\begin{eqnarray*}
\mathcal{D}(\overline{S})&=&\Bigg\{ f=\sum_{n=0}^\infty{c_n}\mathsf{e}_n  :   \{c_n\}_{n=1}^\infty , \ \{n^{2\beta}c_n\}_{n=1}^\infty\in{\ell_2(\mathbb{N})} \\ 
&\mbox{and }& \sum_{n=1}^\infty\frac{n^{\beta}c_n}{n^{\alpha}} \ \mbox{converges and} \ c_0=\sum_{n=1}^\infty\frac{n^{\beta}c_n}{n^{\alpha}} \Bigg\}.
\end{eqnarray*}
Since $(\overline{S}f,f)=\sum_{n=1}^\infty{n^{2\beta}}|c_n|^2$, the operator $\overline{S}$ is positive. 
Using item (iii) of Theorem \ref{new31} we show that the semi-regular sequence $\{\phi_n\}$ is a GRS
for $\alpha>\frac{1}{2}$. To that end, it suffices  to verify the implication \eqref{KKK1}.

Let ${f_m=\sum_{n=0}^\infty{c_n^m}\mathsf{e}_n}$ be a sequence of elements $f_m\in\mathcal{D}(\overline{S})$ satisfying
\eqref{KKK1}. Then 
$$
\lim_{m\to\infty}(\overline{S}f_m, f_m)=\lim_{m\to\infty}\sum_{n=1}^\infty{n^{2\beta}}|c_n^m|^2=\|\{n^{\beta}c_n^m\}\|^2_{\ell_2(\mathbb{N})}=0
$$
and, since $\{{1}/{n^{\alpha}}\}\in\ell_2(\mathbb{N})$ for $\alpha>\frac{1}{2}$,  
$$
g=\lim_{m\to\infty}f_m=\lim_{m\to\infty}\sum_{n=1}^\infty{c_n^m}\mathsf{e}_n+ \lim_{m\to\infty}\sum_{n=1}^\infty\frac{n^{\beta}c_n^m}{n^{\alpha}}\mathsf{e}_0=
\lim_{m\to\infty}(\{n^{\beta}c_n^m\}, \{1/{n^{\alpha}}\})_{\ell_2(\mathbb{N})}\mathsf{e}_0=0
$$
that justifies the  implication \eqref{KKK1}.

\subsection{$J$-orthonormal sequences of the first/second type}\label{sec3.2}
Let a sequence of real numbers $\{\alpha_{k}\}_{k=0}^{\infty}$ satisfy the conditions
\begin{equation}\label{AGH88}
0\leq\alpha_{0}<\alpha_{1}<\alpha_{2}\ldots, \qquad  \lim_{k\to\infty}\alpha_k=\infty
\end{equation}
and let  $\{\mathsf{e}_n\}_{n=0}^{\infty}$  be an ONB of $\cH$ such that  $J\mathsf{e}_n=(-1)^n\mathsf{e}_n$.
 
 Each pair of orthonormal vectors $\{\mathsf{e}_{2k}, \mathsf{e}_{2k+1}\}_{k=0}^\infty$ can be identified with $\mathbb{C}^2$ assuming that
 \begin{equation}\label{AGH44}
U\mathsf{e}_{2k}=\left[\begin{array}{c}
 1 \\
 0 \end{array}\right], \qquad  U\mathsf{e}_{2k+1}=\left[\begin{array}{c}
 0 \\
 1 \end{array}\right].
 \end{equation}
 
 The operator $U$ is an isometric mapping of the space $\cH_k=\mbox{span}\{\mathsf{e}_{2k}, \mathsf{e}_{2k+1}\}$ onto
  $\mathbb{C}^2$ and $UJ=\sigma_3{U}$, where $\sigma_3=\left[\begin{array}{cc}
  1 & 0 \\
  0 & -1 
  \end{array}\right]$.  Since $\cH=\sum_{k=0}^\infty\oplus\cH_k$, the operator $U$ can be extended to 
  the isometric mapping of $\cH$ onto the Hilbert space $\mathbb{H}$ of infinitely many copies of $\mathbb{C}^2$:
   $\mathbb{H}=\sum_{k=0}^\infty\oplus\mathbb{C}^2$.  In the space $\mathbb{H}$, we define
    self-adjoint operators
   \begin{equation}\label{AGH51}
 \mathbb{Q}=2\sum_{k=0}^\infty\oplus\alpha_k\sigma_1, \quad e^{\mathbb{-Q}}=\sum_{k=0}^\infty\oplus{e^{-2\alpha_k\sigma_1}} \quad e^{\mathbb{Q}/2}=\sum_{k=0}^\infty\oplus{e^{\alpha_k\sigma_1}}, \quad 
 \mathbb{J}=\sum_{k=0}^\infty\oplus\sigma_3
\end{equation}
 where $\sigma_1=\left[\begin{array}{cc}
  0 & 1 \\
  1 & 0 
  \end{array}\right].$  Theirs unitary equivalent copies in $\cH$ are:  
  \begin{equation}\label{AGH52}
  Q=U^{-1}\mathbb{Q}U,  \quad e^{-Q}=U^{-1}e^{-\mathbb{Q}}U, \quad e^{Q/2}=U^{-1}e^{\mathbb{Q}/2}U, \quad J=U^{-1}{\mathbb{J}}U.
  \end{equation}
By the construction, $Q$ anticommutes with $J$:  $JQ=-QJ$. 

 Consider vectors $\{\phi_n\}_{n=0}^\infty$ defined by the formulas:
\begin{eqnarray}\label{AGH101}
\phi_{2k}=\cosh\alpha_k\mathsf{e}_{2k}+\sinh\alpha_k{\mathsf{e}_{2k+1}},  \quad k=0,1,\dots \nonumber \vspace{3mm} \\ 
\phi_{2k+1}=\frac{c_k}{\sqrt{\mu_{2k+1}}}\sum_{n=0}^{\infty}\frac{\chi_n\cosh\alpha_n}{1-\mu_{2k+1}\cosh^2\alpha_n}(\cosh\alpha_n\mathsf{e}_{2n+1}+\sinh\alpha_n\mathsf{e}_{2n}),
\end{eqnarray}
where $\{\chi_n\}_{n=0}^\infty$ is a vector from ${\ell_2(\mathbb{N})}$ such that $\chi_n\not=0$; the set
of numbers $0<\mu_{1}<\mu_{3}<\mu_{5}\ldots<1$ are roots of the equation 
\begin{equation}\label{AGH12}
\sum_{n=0}^{\infty}\frac{|\chi_n\cosh\alpha_n|^2}{1-\mu\cosh^2\alpha_n}=0
\end{equation} 
and 
\begin{equation}\label{AGH125}
c_k=\left({\sum_{n=0}^{\infty}\frac{|\chi_n|^2\cosh^4\alpha_n}{(1-\mu_{2k+1}\cosh^2\alpha_n)^2}}\right)^{-\frac{1}{2}}, \quad  k=0,1\ldots
\end{equation}

\begin{ttt}\label{AGH48}
Let the sequences $\{\alpha_{n}\}$ and $\{\chi_{n}\}$ satisfy the conditions above and let the sequence
$\{\chi_n\cosh^2\alpha_n\}$ do not belong to ${\ell_2(\mathbb{N})}$.
Then the vectors $\phi_n$ determined by \eqref{AGH101} form  a complete $J$-orthonormal sequence $\{\phi_n\}_{n=0}^\infty$ of the 
first type if
$\{\chi_n\cosh\alpha_n\}\not\in{\ell_2(\mathbb{N})}$ and of the
second type if
$\{\chi_n\cosh\alpha_n\}\in{\ell_2(\mathbb{N})}$. 

For the first type sequence $\{\phi_n\}$  the formula $\phi_n=e^{Q/2}e_n$ holds where $Q$ and $e^{Q/2}$ are determined by \eqref{AGH52} and an
ONB $\{e_n\}$ has the form
\begin{equation}\label{END1} 
  e_{2k}=\mathsf{e}_{2k}, \qquad e_{2k+1}=
  \frac{c_k}{\sqrt{\mu_{2k+1}}}\sum_{n=0}^{\infty}\frac{\chi_n\cosh\alpha_n}{1-\mu_{2k+1}\cosh^2\alpha_n}\mathsf{e}_{2n+1}.
 \end{equation} 

For the second type sequence such a choice of $Q$ and  $\{e_n\}$ is impossible
because the orthonormal system \eqref{END1} is not dense in $\cH$. 
A suitable operator $Q$ can be chosen as $Q=-\ln{A_F}$, where 
$A_F$ is the Friedrichs extension of the symmetric operator $S$ acting as $S\phi_n=(-1)^nJ\phi_n$ on vectors $\phi_n$ and extended 
onto $\mathcal{D}(S)=\mbox{span}\{\phi_n\}$ by the linearity. 
\end{ttt}
The proof of Theorem \ref{AGH48} is given in Section \ref{sec6}.

Let us consider a particular case assuming that
$$
\tanh\alpha_n=\sqrt{\frac{n}{n+1}}, \qquad  \chi_n=\frac{1}{(n+1)^{\frac{\delta+1}{2}}}, \quad  n\geq{0},
$$ 
and $0<\delta\leq{2}$ (the condition $0<\delta$ guarantees that $\{\chi_n\}\in{\ell_2(\mathbb{N})}$ while $\delta\leq{2}$ ensures that
 $\{\chi_n\cosh^2\alpha_n\}\not\in{\ell_2(\mathbb{N})}$).  
Then the root equation \eqref{AGH12} takes the form
\begin{equation}\label{eqsum}
\sum_{n=1}^{\infty}\frac{1}{n^\delta}\cdot\frac{1}{1-n\mu}=0,
\end{equation}
$c_k=\left({\sum_{n=1}^{\infty}\frac{n^{1-\delta}}{(1-\mu_{2k+1}n)^2}}\right)^{-\frac{1}{2}}$,
and the sequence $\{\phi_n\}_{n=0}^\infty$:
\begin{eqnarray*}
 \phi_{2k}=\sqrt{k+1}\mathsf{e}_{2k}+\sqrt{k}{\mathsf{e}_{2k+1}},  \quad k=0, 1, 2\dots  \vspace{3mm} \\ 
\phi_{2k+1}=\frac{c_k}{\sqrt{\mu_{2k+1}}}\sum_{n=1}^{\infty}\frac{1}{n^{\delta/2}}\cdot\frac{1}{1-n\mu_{2k+1}}(\sqrt{n}\mathsf{e}_{2n-1}+\sqrt{n-1}\mathsf{e}_{2n-2}), 
\end{eqnarray*}
turns out to be the  first kind if $0<\delta\leq{1}$  and the second kind if $1<\delta\leq{2}$.

Figure $1$ contains a numerical localization of the first 5 roots for of (\ref{eqsum}).
\begin{figure}
\centering
\includegraphics[width=15cm]{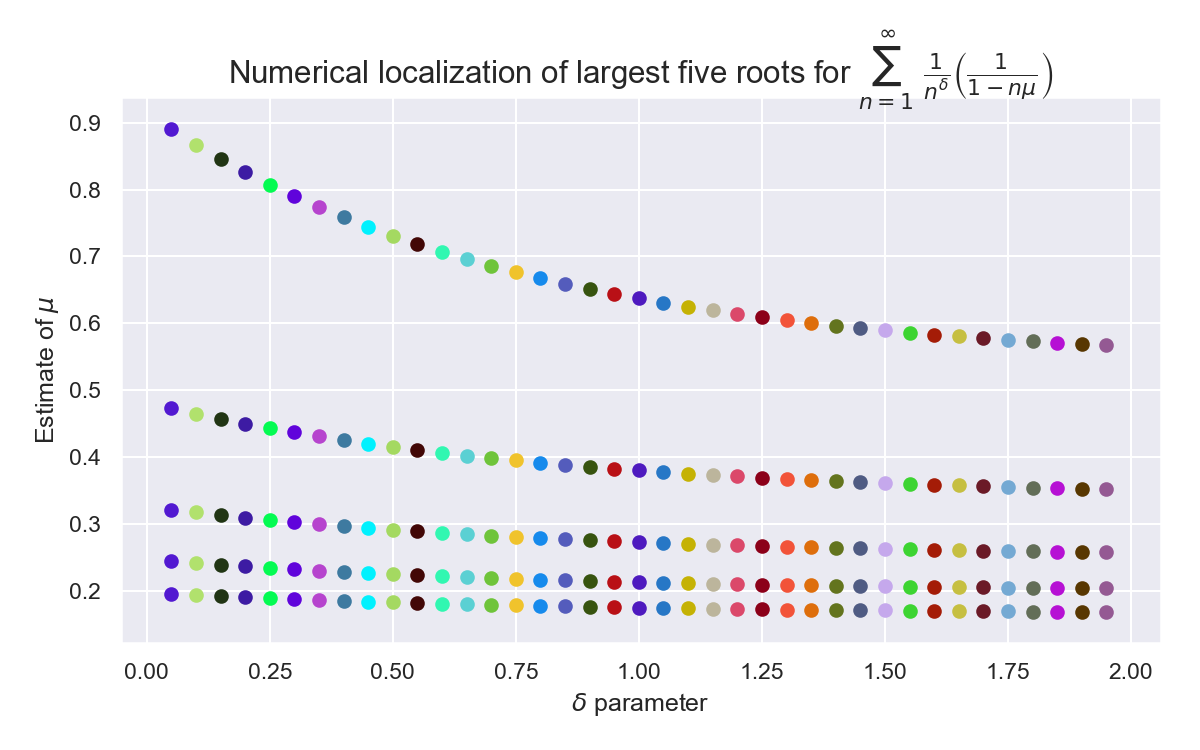}
\caption{First five roots to the equation (\ref{eqsum}) for different
δ parameter value.}
\end{figure}

\section{Appendix: the proof of Theorem \ref{AGH48}}\label{sec6}
We refer \cite[Subsection III.1]{Kuzhel} and \cite{KKS} for results of the Krein space theory which are necessary for our exposition.
\subsection{Preliminaries}
An operator $T$ defined by the formula
\begin{equation}\label{AGH67b}
    Te_{2k}=\tanh\alpha_k\mathsf{e}_{2k+1}, \qquad  Te_{2k+1}=\tanh\alpha_k\mathsf{e}_{2k}, \qquad k=0,1,2\dots
  \end{equation}  
  on the ONB $\{\mathsf{e}_n\}_{n=0}^\infty$ and extended onto $\cH$ by the linearity is a self-adjoint strong contraction 
  (`strong' means that  $\|Tf\|<\|f\|$ for non-zero $f$). Moreover 
\begin{equation}\label{AAA4}
JT=-TJ,
\end{equation}
 since $J\mathsf{e}_n=(-1)^n\mathsf{e}_n$. 
The properties of $T$ allow one to define $J$-orthogonal maximal positive $\mathfrak L_+$ and maximal negative $\mathfrak L_-$ subspaces 
of the Krein space $(\cH, [\cdot,\cdot])$ \cite[Lemma 2.2]{KKS}:
\begin{equation}\label{AGH36}
\mathfrak L_+=(I+T)\cH_+,  \qquad
\mathfrak L_-=(I+T)\cH_-,
\end{equation}
where $\cH_+$ and $\cH_-$ are the closure (in $\cH$) of $\mbox{span}\{\mathsf{e}_{2k}\}_{k=0}^\infty$ and $\mbox{span}\{\mathsf{e}_{2k+1}\}_{k=0}^\infty$, respectively.

Consider a vector $\chi=\sum_{k=0}^\infty\chi_k\mathsf{e}_{2k+1}$, where $\{\chi_k\}$ belongs to ${\ell_2(\mathbb{N})}$.
By the construction, $\chi\in\cH_{-}$ and 
\begin{equation}\label{AGH8}
\mathfrak L_-^0=\{(I+T)g ,\ g\in M_- \},  \quad \mbox{where} \quad M_-=\{g\in \cH_-,\ (g, \chi)=0 \}
\end{equation}
is a subspace of the maximal negative space $\mathfrak L_-$ defined by \eqref{AGH36}.

The $J$-orthogonal sum $\mathfrak L_+{\dot+}\mathfrak L_-$
is dense in $\cH$ since $\mathfrak L_\pm$ are maximal subspaces. However,  we can not state that the set
$\mathfrak L_+{\dot+}\mathfrak L_-^0$ remains  dense  in $\cH$ since $\mathfrak L_-^0$ is a proper subspace of $\mathfrak L_-$.

\begin{lelele}\label{AGH19}
The sum $\mathfrak L_+{\dot+}\mathfrak L_-^0$ is dense in $\cH$ if and only if  $\left\{\chi_n\cosh^2\alpha_n\right\}\not\in\ell_2(\mathbb{N})$.
\end{lelele}
\begin{proof}
Assume that $h\in\cH$ is orthogonal to $\mathfrak L_+{\dot+}\mathfrak L_-^0$. 
Then $Jh$ is orthogonal to $\mathfrak L_+$ with respect to the indefinite inner product $[\cdot,\cdot]$. 
Since $\mathfrak L_-$ is the $J$-orthogonal complement of $\mathfrak L_+$,  the vector $Jh$ belongs to $\mathfrak L_-$. 
By \eqref{AGH36},  $Jh=(I+T)f$,  where $f\in\cH_-$ and
$h=J(I+T)f=(I-T)Jf=-(I-T)f.$  

By the assumption, $h$ is also orthogonal to $\mathfrak L_-^0$.  In view of  \eqref{AGH8} this means that
$0=(h, (I+T)g)=-((I-T)f, (I+T)g)=-((I-T^2)f, g)$ for all $g\in M_-$.
Therefore, without loss of generality we can assume that
$(I-T^2)f=\chi$.  Here, $f=\sum_{k=0}^{\infty}{f_k \mathsf{e}_{2k+1}}$ since $f\in\cH_-$.
In view of \eqref{AGH67b}, 
$$
(I-T^2)f=\sum_{k=0}^{\infty}{(1-\tanh^2\alpha_k)f_k\mathsf{e}_{2k+1}}=\sum_{k=0}^{\infty}\frac{f_k}{\cosh^2\alpha_k}\mathsf{e}_{2k+1}=\sum_{k=0}^{\infty}{\chi_k\mathsf{e}_{2k+1}}=\chi.
$$
Therefore, $f_k=\chi_k\cosh^2\alpha_k$ and we arrive at the conclusion that 
$\mathfrak L_+\dot+\mathfrak L_-^0$ is a non-dense set in $\cH$ if and only if 
 $\{\chi_k\cosh^2\alpha_k\}^\infty_{k=0}\in\ell_2(\mathbb N)$.  
\end{proof}

\subsection{Complete $J$-orthonormal sequence $\{\phi_n\}$}
\begin{lelele}\label{AGH20}
If $\{\chi_n\cosh^2\alpha_n\}\not\in{\ell_2(\mathbb{N})}$, then the vectors $\{\phi_n\}$ 
defined by \eqref{AGH101} form a complete $J$-orthonormal sequence in $\cH$.
\end{lelele}
\begin{proof}
The vectors $\{\phi_{2k}\}_{k=0}^\infty$ in \eqref{AGH101} are $J$-orthonormal because
[$\phi_{2k}, \phi_{2k'}]=0$ for $k\not=k'$ and
$[\phi_{2k}, \phi_{2k}]=\cosh^2\alpha_k-\sinh^2\alpha_k=1.$
Moreover, in view of of \eqref{AGH67b},  the vectors $\{\phi_{2k}\}$ can be presented as 
 $\phi_{2k}=\cosh\alpha_k(I+T)\mathsf{e}_{2k}$. This relation
and \eqref{AGH36} imply that the closure of $\mbox{span}\{\phi_{2k}\}_{k=0}^{\infty}$ coincides with $\mathfrak L_+$.

By virtue of \eqref{AGH67b}, 
$$
(I-T^2)f=\sum_{k=0}^{\infty}{(1-\tanh^2\alpha_k)f_k\mathsf{e}_{2k+1}} \quad \mbox{for all} \quad f=\sum_{k=0}^{\infty}{f_k \mathsf{e}_{2k+1}}\in\cH_-.
$$
This relation yields that $I-T^2$ is a compact operator in $\cH_-$, since $\lim_{k\to\infty}(1-\tanh^2\alpha_k)=0$, 
see \cite[problem 132]{Halmoshbook}. Therefore,  $P_{M_-}(I-T^2)P_{M_-}$, where $P_{M_-}$ is an orthogonal projection in $\cH_-$ onto 
the subspace $M_-$ defined in  \eqref{AGH8},  is
a self-adjoint compact operator in  $M_-$.  This implies the existence of an ONB   $\{\gamma_{2k+1}\}$ of $M_-$ which is formed
by eigenfunctions of $P_{M_-}(I-T^2)P_{M_-}$.  Let $\{\mu_{2k+1}\}_{k=0}^\infty$ be the corresponding eigenvalues, i.e.,
$P_{M_-}(I-T^2)\gamma_{2k+1}=\mu_{2k+1}\gamma_{2k+1}$. Since $I-T^2$ is a positive contraction, we can state that
 $0<\mu_{2k+1}<{1}$ and $\lim_{k\to\infty}\mu_{2k+1}=0$.

Denote
\begin{equation}\label{AGH111}
\phi_{2k+1}=\frac{1}{\sqrt{\mu_{2k+1}}}(I+T)\gamma_{2k+1}, \qquad k=0,1\ldots
\end{equation}
The vectors $\{\phi_{2k+1}\}_{k=0}^\infty$ are $J$-orthonormal because
$$
[\phi_{2k+1}, \phi_{2k'+1}]=-\frac{((I-T^2)\gamma_{2k+1}, \gamma_{2k'+1})}{\sqrt{\mu_{2k+1}\mu_{2k'+1}}}
=-\sqrt{\frac{\mu_{2k+1}}{\mu_{2k'+1}}}(\gamma_{2k+1}, \gamma_{2k'+1})=-\delta_{kk'}.
$$
Moreover,  in view of \eqref{AGH8},  $\phi_{2k+1}\in{\mathfrak L}_-^0$ and the closure of $\mbox{span}\{\phi_{2k+1}\}_{k=0}^\infty$ 
coincides with ${\mathfrak L}_-^0$. Applying now  Lemma \ref{AGH19} we arrive at the conclusion that, 
for the case $\{\chi_n\cosh^2\alpha_n\}\not\in{\ell_2(\mathbb{N})}$,
the $J$-orthonormal sequence $\{\phi_n\}_{n=0}^\infty$, where $\{\phi_{2k}\}_{k=0}^\infty$ and $\{\phi_{2k+1}\}_{k=0}^\infty$ are defined by 
\eqref{AGH101} and \eqref{AGH111}, respectively, is complete in $\cH$.

To finish the proof of Lemma \ref{AGH20} it suffices to show that  the formulas
\eqref{AGH101} and \eqref{AGH111} determine the same 
vectors $\{\phi_{2k+1}\}_{k=0}^\infty$. To do that,  we describe 
the eigenvalues $\mu_{2k+1}$ and the normalized eigenfunctions $\gamma_{2k+1}$ of the equation
 \begin{equation}\label{AGH10}
P_{M_-}(I-T^2)g=\mu{g}, \qquad g\in{M_-}.
\end{equation}

In view of \eqref{AGH8},  the  condition $g\in{M_-}$  means that 
\begin{equation}\label{AGH11b}
\sum_{k=0}^{\infty}g_k\overline{\chi}_k=0, \quad \mbox{where} \quad g=\sum_{k=0}^{\infty}{g_k \mathsf{e}_{2k+1}}.
\end{equation}

Let $f=\sum_{k=0}^{\infty}{f_k \mathsf{e}_{2k+1}}$ be an arbitrary element of $\cH_-$. Then
\begin{equation}\label{AGH12b}
P_{M_-}f=\sum_{k=0}^{\infty}{(f_k-\alpha[f]\chi_k)\mathsf{e}_{2k+1}}, \qquad \alpha[f]=\frac{1}{\|\chi\|^2}\sum_{k=0}^{\infty}f_k\overline{\chi}_k.
\end{equation}
 
 Using \eqref{AGH12b} we rewrite  \eqref{AGH10} as
$$
 P_{M_-}\sum_{k=0}^{\infty}(1-\tanh^2\alpha_k)g_k\mathsf{e}_{2k+1}=\sum_{k=0}^{\infty}[(1-\tanh^2\alpha_k)g_k-\alpha[(I-T^2)g]\chi_k]\mathsf{e}_{2k+1}=\mu\sum_{k=0}^{\infty}{g_k \mathsf{e}_{2k+1}} 
$$
that implies
\begin{equation}\label{AGH62}
(1-\tanh^2\alpha_k-\mu)g_k=\alpha[(I-T^2)g]\chi_k, \qquad k=0,1,\ldots 
\end{equation}
It is important that $\alpha[(I-T^2)g]\not=0$ in \eqref{AGH62}. Indeed, if $\alpha[(I-T^2)g]=0$, then 
$(1-\tanh^2\alpha_k-\mu)g_k=0$ for all 
$k$. Due to conditions imposed on $\alpha_n$ in \eqref{AGH88}, there exists a unique $k'$ such that $
g_{k'}\not=0$ and $\mu=1-\tanh^2\alpha_{k'}$.  This means that $g=g_{k'}\mathsf{e}_{2k'+1}$ belongs to $M_-$. 
The last fact is impossible because $0=<g, \chi>=g_{k'}\overline{\chi}_{k'}\not=0$ (we recall that $\chi_n\not=0$ for all $n$ by the assumption). 
The obtained contradiction shows that  $\alpha[(I-T^2)g
]\not=0$. This means that $(1-\tanh^2\alpha_k-\mu)g_k\not=0$ and \eqref{AGH62} can be rewritten as
$$
g_k=\alpha[(I-T^2)g]\frac{\chi_k\cosh^2\alpha_k}{1-\mu\cosh^2\alpha_k},  \qquad k=0,1,2\dots
$$

The corresponding solution $g(\mu)=\sum_{k=0}^{\infty}g_k\mathsf{e}_{2k+1}$ of \eqref{AGH10} must be in $M_-$. 
By virtue of \eqref{AGH11b},  $g(\mu)$ belongs to $M_-$ if and only if $\mu$ is the root of \eqref{AGH12}.
The equation \eqref{AGH12} has infinitely many roots $0<\mu_{1}<\mu_{3}<\ldots<\mu_{2k+1}\ldots<1$ 
that coincide with eigenvalues of $P_{M_-}(I-T^2)P_{M_-}$. The eigenfunctions of $P_{M_-}(I-T^2)P_{M_-}$ corresponding to $\mu_{2k+1}$ have the form
$$
g(\mu_{2k+1})=\alpha[(I-T^2)g(\mu_{2k+1})]\sum_{n=0}^{\infty}\frac{\chi_n\cosh^2\alpha_n}{1-\mu_{2k+1}\cosh^2\alpha_n}\mathsf{e}_{2n+1}, \qquad k=0,1\ldots
$$
Then
$$
\gamma_{2k+1}=\frac{g(\mu_{2k+1})}{\|g(\mu_{2k+1})\|}={c_k}\sum_{n=0}^{\infty}\frac{\chi_n\cosh^2\alpha_n}{1-\mu_{2k+1}\cosh^2\alpha_n}\mathsf{e}_{2n+1}, \qquad k=0,1\ldots,
$$
where the normalizing factor $c_k$ is defined in \eqref{AGH125}. By the construction, $\{\gamma_{2k+1}\}_{k=0}^\infty$ is an ONB of $M_-$.
Substituting the obtained expression for $\gamma_{2k+1}$ into \eqref{AGH111} and taking \eqref{AGH67b} into account, we obtain
the vectors $\{\phi_{2k+1}\}$ from \eqref{AGH101}.
\end{proof}

\begin{lelele}\label{NEWWW}
If  $\{\chi_n\cosh\alpha_n\}\not\in\ell_2(\mathbb{N})$,
then the sequence $\{\phi_n\}$ is
of the first type. The corresponding operator $Q$ in Definition \ref{d1} is defined  by
 \eqref{AGH52} while ONB $\{e_n\}$ has the form \eqref{END1}.
\end{lelele}
\begin{proof}
If $\{\chi_n\cosh\alpha_n\}\not\in{\ell_2(\mathbb{N})}$, then 
$\{\chi_n\cosh^2\alpha_n\}\not\in{\ell_2(\mathbb{N})}$ and, by  Lemma \ref{AGH20},   
$\{\phi_n\}$ is  a complete $J$-orthonormal sequence.
In view of \eqref{new4}, the operator $S$ defined by \eqref{new14} acts
as $S\phi_n=(-1)^nJ\phi_n$. On the other hand, taking the relation
$e^{-2\alpha_k\sigma_1}=\cosh{2\alpha_k}\sigma_0-\sinh{2\alpha_k}\sigma_1$ 
into account, we directly verify that $e^{-Q}\phi_n=(-1)^nJ\phi_n$, 
where $e^{-Q}$ is defined by \eqref{AGH52}.  Therefore, $e^{-Q}$ is a positive self-adjoint
extension of $S$.  Denote $e_n=e^{-Q/2}\phi_n$.  In view of \eqref{AGH51},
$$
e^{-Q/2}=U^{-1}\sum_{k=0}^\infty\oplus[\cosh{\alpha_k}\sigma_0-\sinh{\alpha_k}\sigma_1]U=U^{-1}\sum_{k=0}^\infty\oplus\left[\begin{array}{cc}
\cosh\alpha_k & -\sinh\alpha_k \\
-\sinh\alpha_k & \cosh\alpha_k 
\end{array}\right]U.
$$
This expression, \eqref{AGH44}, and \eqref{AGH101} allow one to calculate $\{e_n\}$ precisely, as \eqref{END1}.

 By analogy with the proof of Theorem  \ref{new31} we
obtain that $\{e_n\}$ is an orthonormal sequence in $\cH$. Moreover,
 $\{e_n\}$  is an ONB if and only if $\{\phi_n\}$ is complete in
the Hilbert space  $\cH_{-Q}$.
Below we show that the completeness of $\{\phi_n\}$ in $\cH_{-Q}$ is equivalent 
to the condition $\{\chi_n\cosh\alpha_n\}\not\in\ell_2(\mathbb{N})$.

We begin with the remark that
\begin{equation}\label{AGH126}
e^{-Q}=(I-T)(I+T)^{-1},
\end{equation}
 where $T$ is determined by \eqref{AGH67b}. 
 Indeed, since  the subspaces $\cH_k=\mbox{span}\{\mathsf{e}_{2k}, \mathsf{e}_{2k+1}\}$ are invariant with respect to $T$ and $U$ satisfies
\eqref{AGH44}, we get that $UTU^{-1}|_{\cH_k}$ acts as the multiplication by $\tanh\alpha_k\sigma_1$ in $\mathbb{C}^2$ and
$U(I-T)(I+T)^{-1}U^{-1}|_{\cH_k}$ coincides with  
$$
\left[\begin{array}{cc}
\cosh\alpha_k & -\sinh\alpha_k \\
-\sinh\alpha_k & \cosh\alpha_k 
\end{array}\right]^2=\left[\begin{array}{cc}
\cosh2\alpha_k & -\sinh2\alpha_k \\
-\sinh2\alpha_k & \cosh2\alpha_k 
\end{array}\right]=(\cosh2\alpha_k\sigma_0-\sinh2\alpha_k\sigma_1)^2=e^{-2\alpha_k\sigma_1}.
$$
 This relation and the decomposition $\cH=\sum_{k=0}^\infty\oplus\cH_k$ justify \eqref{AGH126}.

The formulas \eqref{AGH36} and \eqref{AGH126} lead to the conclusion that 
$\mathcal{D}(e^{-Q})=\mathfrak L_+{\dot+}\mathfrak L_-$, where the subspaces $\mathfrak L_\pm$
are orthogonal with respect to the inner product \eqref{new1}. Therefore, the space $\cH_{-Q}$  has the decomposition
 $$
 \cH_{-Q}=\widehat{\mathfrak L}_+{\dot+}\widehat{\mathfrak L}_-
 $$ 
where the subspaces $\widehat{\mathfrak L}_\pm$ are the completions of linear manifolds ${\mathfrak L}_\pm$
in $\cH_{-Q}$.

To prove the completeness of $\{\phi_n\}$ in  $\cH_{-Q}$ we
note that  $\{\phi_{2k}\}$ is a complete set in $\widehat{{\mathfrak L}}_+$. This  fact can be justified as follows: due to the proof of Lemma \ref{AGH20},
 $\mbox{span}\{\phi_{2k}\}$ is  dense in the subspace $\mathfrak L_+$ of $\cH$. 
In view of \eqref{AAA4}, \eqref{AGH36}, and \eqref{AGH126}, 
$$
\|f\|_{-Q}^2=(e^{-Q}f,f)=((I-T)x_+, (I+T)x_+)=[f,f]\leq\|f\|^2
$$ 
for each $f=(I+T)x_+\in\mathfrak L_+$. Therefore, each $f\in\mathfrak L_+$ can be approximated by
vectors from $\mbox{span}\{\phi_{2k}\}$ with respect to the norm $\|\cdot\|_{-Q}$. Since $\widehat{{\mathfrak L}}_+$ is the completion of 
$\mathfrak L_+$ in $\cH_{-Q}$,  the set $\{\phi_{2k}\}$ is  complete  in $\widehat{{\mathfrak L}}_+$.
 
 Similar arguments and the fact that $\mbox{span}\{\phi_{2k+1}\}$  is dense in  the subspace $\mathfrak L_-^0$ 
 lead to the conclusion that each vector $f\in\mathfrak L_-^0$ can be approximated by vectors of $\mbox{span}\{\phi_{2k+1}\}$ 
 with respect to  $\|\cdot\|_{-Q}$.
Therefore, in order to  proof the completeness of $\{\phi_{2k+1}\}$ in $\widehat{{\mathfrak L}}_-$
it suffices to find when $\mathfrak L_-^0$ turns out to be dense in $\widehat{{\mathfrak L}}_-$ with respect to $\|\cdot\|_{-Q}$.

Let $h\in\widehat{{\mathfrak L}}_-$ be orthogonal to $\mathfrak L_-^0$ in $\cH_{-Q}$. Since $\mathfrak L_-$  is dense in $\widehat{{\mathfrak L}}_-$
we can approximate $h$ by a sequence $\{f_n\}$, where $f_n\in\mathfrak L_-$.  
In view of \eqref{new1}, the sequence $\{e^{-Q/2}f_n\}$ is
fundamental in $\cH$ and, hence, $\lim_{n\to\infty}e^{-Q/2}f_n=f\in\cH$.  Due to \eqref{AGH36},  $f_n=(I+T)x_{-}^n$, where $x_-^n\in\cH_-$. 
Moreover,  $e^{-Q/2}=[(I-T)(I+T)^{-1}]^{1/2}$  in view of \eqref{AGH126}. 
This means that 
$$
e^{-Q/2}f_n=[(I-T)(I+T)^{-1}]^{1/2}(I+T)x_{-}^n=(I-T^2)^{1/2}x_-^n
$$ and, since $(I-T^2)^{1/2}$ leaves $\cH_\pm$ invariant, 
$f=\lim_{n\to\infty}e^{-Q/2}f_n=\lim_{n\to\infty}(I-T^2)^{1/2}x_-^n=f\in\cH_-$. 
 On the other hand, for each vector $(I+T)g\in\mathfrak L_-^0$, the relation
$e^{-Q/2}(I+T)g=(I-T^2)^{1/2}g$ holds.  After such kind of auxiliary work we obtain:
\begin{eqnarray*}
0=(h, (I+T)g)_{-Q}=\lim_{n\to\infty}(f_n, (I+T)g)_{-Q}=\lim_{n\to\infty}(e^{-Q/2}f_n, e^{-Q/2}(I+T)g) & = & \\
(f, (I-T^2)^{1/2}g)=((I-T^2)^{1/2}f, g) \quad \mbox{for all} \quad g\in M_-.
\end{eqnarray*}
Therefore, without loss of generality we can assume that $(I-T^2)^{1/2}f=\chi$. 
Reasoning by analogy with
the final part of the proof of Lemma \ref{AGH19},  we obtain that   
 $\mathfrak L_-^0$ is dense in $\widehat{{\mathfrak L}}_-$ if and only if 
$\{\chi_n\cosh\alpha_n\}\not\in\ell_2(\mathbb{N})$.  This relation guarantees 
the completeness of  $\{\phi_n\}$ in $\cH_{-Q}$.
\end{proof}

\subsection{The proof of Theorem \ref{AGH48}}

The implication $`\{\chi_n\cosh\alpha_n\}\not\in{\ell_2(\mathbb{N})}  \to  \{\phi_n\}$ is a first type sequence' 
was proved in Lemma \ref{NEWWW}.  Let us assume  that $\{\phi_n\}$ is  first type, i.e., there exists
a self-adjoint operator $Q'$ anti-commuting with $J$ and such that  $\phi_n=e^{Q'/2}e_n'$, where $\{e_n'\}$
 is an ONB of $\cH$.  

Denote $\cC=Je^{-Q'}$. Since $Q'$ anti-commutes with $J$, the operator $\cC$ satisfies the relation 
$\cC^2f=f$ for $f\in\mathcal{D}(\cC)$ and $J\cC=e^{-Q'}$ is a positive self-adjoint operator in $\cH$.
In view of \cite[Theorem 6.2.3]{BGSZ},  there exists $J$-orthonormal maximal positive $\mathfrak{L}_+'$ and maximal negative $\mathfrak{L}_-'$ subspaces  of the Krein space 
$(\cH, [\cdot,\cdot])$ which uniquely characterize $\cC$ in the following way: $\cC{f_+}=f_+$ and $\cC{f_-}=f_-$  for $f_\pm\in\mathfrak{L}_\pm'$.

Since $e^{-Q'}$ is an  extension of $S$, we obtain $\cC\phi_n=Je^{-Q'}\phi_n=JS\phi_n=(-1)^{n}\phi_n$.
Therefore, the operator $\cC$ acts as the identity operator on elements of the subspace $\mathfrak{L}_+$ defined by \eqref{AGH36}
 (since $\mbox{span}\{\phi_{2k}\}_{k=0}^{\infty}$ is dense in $\mathfrak{L}_+$). This yields
that $\mathfrak{L}_+'=\mathfrak{L}_+$ and, moreover $\mathfrak{L}_-'=\mathfrak{L}_-$ since  the maximal negative
subspace $\mathfrak{L}_-'$ is determined uniquelly as $J$-orthogonal complement of  $\mathfrak{L}_+'=\mathfrak{L}_+$.
 We obtain that the $J$-orthogonal sum $\mathfrak{L}_+\dot{+}\mathfrak{L}_-$ determines two operators $Je^{-Q'}$ and
$Je^{-Q}$. Applying again \cite[Theorem 6.2.3]{BGSZ}, we conclude that $Q'=Q$, where $Q$ is determined by \eqref{AGH52}.
In this case, $e_n'=e^{-Q'/2}\phi_n=e^{-Q/2}\phi_n=e_n$,  where $\{e_n\}$ is determined by \eqref{END1}.
Therefore, $\{e_n\}$ is an ONB of $\cH$ that, in view of the proof of Lemma  \ref{NEWWW} is equivalent to the condition 
$\{\chi_n\cosh\alpha_n\}\not\in{\ell_2(\mathbb{N})}$.   The inverse implication  `first type sequence $\{\phi_n\} \to \{\chi_n\cosh\alpha_n\}\not\in{\ell_2(\mathbb{N})}$ ' is proved. 

If $\{\phi_n\}$ is second type, the choice of $Q$  as in \eqref{AGH52} leads to the non-complete  orthonormal sequence
\eqref{END1} in $\cH$.  Trying to keep the GRS's formula $\phi_n=e^{Q/2}e_n$ we have to use
$Q=-\ln A$, where $A$ is a positive extremal extension of $S$ (without loss of generality, we may assume that $A=A_F$). 
In this case, the ONB $\{e_n\}$ will be different from \eqref{END1}. 
 
\section*{Acknowledgements}
This work was partially supported by the Faculty of Applied Mathematics AGH UST statutory tasks within subsidy of 
Ministry of Science and Higher Education.


\begin{thebibliography}{99}
\bibitem{Ando}  Ando, T., Nishio, K.: Positive self-adjoint operators of positive symmetric operators. Tohoku Math. Journ. {\bf 22}, pp.65--75 (1970)

\bibitem{AK_Arlin}  Arlinski\u{i}, Y. M., Hassi, S., Sebesty\'{e}n, Z.,  de Snoo, H.S.V.: On the class of extremal extensions of a nonnegative operator. In: Recent Advances in Operator Theory and Related Topics the Bela Szokefalvi-Nagy Memorial Volume, Operator Theory: Advances and Applications, Vol. {\bf 127}. Birkh\"{a}user, pp.41--81 (2001)

\bibitem{AT}  Arlinski\u{i}, Y. M., Tsekanovski\u{i}, E.: M. Krein's research on semi-bounded operators, its contemporary developments, and applications.  Oper. Theory Adv. Appl.  Vol {\bf 190}.  Birkh\"{a}user Basel pp.65--112 (2009)


\bibitem{bag2013JMP} Bagarello, F.: More mathematics on pseudo-bosons.  J. Math. Phys. {\bf 54}, 063512 (2013)

\bibitem{BB} Bagarello, F.,  Bellomonte, G.: Hamiltonians defined by biorthogonal sets. J. Phys. A {\bf 50}, 145203 (2017)

 \bibitem{BIT}  Bagarello, F., Inoue, H.,  Trapani, C.:  Biorthogonal vectors, sesquilinear forms and some physical operators.  J. Math. Phys. {\bf 59}, 033506 (2018)
 
\bibitem{BGSZ} Bagarello, F.,  Gazeau, J.-P.,  Szafraniec, F. H., Znojil, M. (eds.): Non-Selfadjoint Operators in Quantum Physics. Mathematical Aspects. J. Wiley \& Sons, Hoboken, New Jersey (2015)
 
\bibitem{Kuzhel} Bagarello, F., Kuzhel, S.: Generalized Riesz systems and orthonormal sequences in Krein spaces.  arXiv:1810.05218 (2018)

\bibitem{Trapani} Bellomonte, G., Trapani, C.: Riesz-like bases in rigged Hilbert spaces. Zeitschr. Anal. Anwend. {\bf 35} pp.243--265 (2016)

\bibitem{Bender} C. M. Bender, C. M. et al:  $\mathcal{PT}$-Symmetry in Quantum and Classical Physics.  World Scientific, Singapore  (2019) https://doi.org/10.1142/q0178 

\bibitem{BenderSI} Bender, C. M., Fring A., G{\"u}nther, U., Jones, H. (eds.): Special issue on quantum physics with non-Hermitian operators.  J. Phys. A {\bf 45}, issue 44 (2012)

\bibitem{Christ} Christensen, O.: An Introduction to Frames and Riesz Bases, Birkh\"{a}user Boston (2003)

\bibitem{Davies}  Davies, E. B.:  Wild spectral behaviour of anharmonic oscillators. Bull. London Math. Soc. {\bf 32}, pp.432--438 (1999)

\bibitem{Inoue1} Inoue, H.: General theory of regular biorthogonal pairs and its physical operators. J. Math. Phys. {\bf 57}, 083511 (2016)
\bibitem{Inoue3} Inoue, H.: Semi-regular biorthogonal pairs and generalized Riesz bases. J. Math. Phys. {\bf 57}, 113502 (2016)

\bibitem{Inoue4} Inoue, H., Takakura, M.:  Regular biorthogonal pairs and pseudo-bosonic operators.  J. Math. Phys. {\bf 57},  083503 (2016)

\bibitem{Inoue} Inoue, H., Takakura, M.: Non-self-adjoint hamiltonians defined by generalized Riesz bases. J. Math. Phys. {\bf 57}, 083505 (2016)

\bibitem{Halmoshbook}  Halmos, P.R.: A Hilbert Space Problem Book. Springer New York (1982)

\bibitem{Heil} Heil, C.:  A Basis Theory Primer. In: Applied and Numerical Harmonic Analysis. Birkh\"{a}user Boston (2011)

\bibitem{KKS} Kamuda, A.,  Kuzhel, S., Sudilovskaja, V.: On dual definite subspaces in Krein space. 

Complex Anal. Oper. Theory (2019)  13: 1011. doi.org/10.1007/s11785-018-0838-x

\bibitem{Kato} Kato, T.: Perturbation Theory for Linear Operators. Springer-Verlag, Berlin (1966)

\bibitem{Krejcirik} Krej\v{c}i\v{r}{\'\i}k, D.,  Siegl, P., Tater, M., Viola, J.: Pseudospectra in non-Hermitian quantum mechanics.  J.  Math. Phys. {\bf 56}, 103513 (2015) doi.org/10.1063/1.4934378

\bibitem{Mit} Mityagin, B., Siegl, P., Viola, J.: Differential operators admitting various rates of spectral projection growth.  J. Funct. Anal. {\bf 272}, pp.3129--3175 (2017)

\bibitem{Most}  Mostafazadeh, A.: Pseudo-Hermitian representation of quantum mechanics. Int. J. Geom. Methods Mod. Phys. {\bf 7}, pp.1191--1306 (2010)

\bibitem{Olew}  Olevskii, A.M.:  On operators generating conditional bases in a Hilbert space.  Math. Notes, {\bf 12}, pp.476--482 (1972) 

\bibitem{Konrad}  Schm\"{u}dgen, K.: Unbounded Self-adjoint Operators on Hilbert Space.  Springer Dordrecht Heidelberg New York London, pp.435 (2012)

\bibitem{Siegl}  Siegl, P., Krej\v{c}i\v{r}{\'\i}k, D.: On the metric operator for the imaginary cubic oscillator. Physical review D.  {\bf 86}, 121702 (2012) doi.org/10.1103/PhysRevD.86.121702






\end{thebibliography}
\end{document}